\documentclass{article}
\usepackage[all]{xy}
\usepackage{amscd,amsmath,amsthm,amssymb}
\usepackage{hyperref}
\newtheorem{theorem}{Theorem}[subsection]
\newtheorem{lemma}[theorem]{Lemma}

\newtheorem{proposition}[theorem]{Proposition}
\newtheorem{corollary}[theorem]{Corollary}
\newcommand{\End}{\mathrm{End}}
\newcommand{\Ext}{\mathrm{Ext}}
\newcommand{\Hom}{\mathrm{Hom}}

\newcommand{\im}{\mathrm{im}}
\newcommand{\Hilb}{\mathrm{Hilb}}
\newcommand{\Hilban}{\mathrm{Hilb}^n _A}
\newcommand{\Hilbfn}{\mathrm{Hilb}^n _F}
\newcommand{\bt}{\beta}
\newcommand{\GL}{\mathrm{GL}}

\newcommand{\NN}{\mathbb{N}}

\newcommand{\CC}{\mathbb{C}}

\newcommand{\A}{\mathcal{A}}

\newcommand{\C}{\mathcal{C}}

\newcommand{\nn}{\mathcal{N}}

\newcommand{\ran}{\mathrm{Rep}_A^n}
\newcommand{\rana}{\ran \times _k \mathbb{A}^n_k}

\newcommand{\uan}{\mathrm{U}_A^n}
\newcommand{\ufn}{\mathrm{U}_F^n}

\newcommand{\op}{\operatorname{op}}
\newcommand{\md}{\operatorname{mod}}

\newcommand{\n}{\noindent}
\theoremstyle{definition}
\newtheorem{definition}[theorem]{Definition}
\newtheorem{example}[theorem]{Example}

\theoremstyle{remark}
\newtheorem{remark}[theorem]{Remark}

\begin{document}
\title{On the smoothness of the scheme of linear representations and the Nori-Hilbert scheme of an associative algebra}

\author{Federica Galluzzi , Francesco Vaccarino \thanks{Work done during a visit to the Department of Mathematics, KTH (Stockholm ,Sweden). Support by the Institut Mittag-Leffler (Djursholm, Sweden)is gratefully acknowledged.
The first author is supported by PRIN 2008 "Geometria delle Variet\`a
Algebriche e dei loro Spazi di Moduli".
The second author is supported by the Wallenberg grant and by PRIN 2009 "Spazi di Moduli e Teoria di Lie".}}

\maketitle

\begin{abstract} Let $k$ be an algebraically closed field and let $A$ be a finitely generated $k-$algebra. We show that the scheme of n-dimensional representations of $A$ is smooth when $A$ is hereditary and coherent. We deduce from this the smoothness of the Nori-Hilbert scheme associated to $A.$

\medskip
\n

\n

\end{abstract}

\section{Introduction}\label{intro}
Let $A$ be a finitely generated  associative $k-$algebra with $k$ algebraically closed field. In this paper we study the scheme $\ran$ of the $n-$dimensional representations of $A.$

According to Victor Ginzburg (see \cite[Chapt. 12]{G1}), it is believed that the noncommutative geometry of an associative algebra $A$ is \textit{approximated} (in certain cases) by the (commutative) geometry of the scheme
$\ran$. 

Following this path it appear relevant to look for conditions on $A$ for $\ran$ to be smooth. It is well-known that $\ran$ is smooth if $A$ is formally smooth. This last condition implies that $A$ is hereditary, that is, of global dimension one. In particualr, if $A$ is finite-dimensional over $k$, the scheme $\ran$ is smooth if and only if $A$ is hereditary.

In this paper we give a sufficient condition for a $k-$point in $\ran$ to be smooth.
Using this condition, we prove that $\ran$ is smooth when $A$ is hereditary and coherent, which means that the kernel of any morphism between finite rank free $A$-bimodules is finitely generated. Note that a formally smooth algebra is not necessarily coherent.Furthermore coherence does not implies hereditariety.  Therefore this is an independent notion. 

From the proved smoothness result, we derive the smoothness of the Nori - Hilbert scheme $\Hilban$  whose $k-$points parameterize left ideals of codimension $n$ of $A.\,$
When $A$ is commutative, this is nothing but the classical Hilbert scheme $\mathrm{Hilb}_X^n$ of the $n-$points on  $X=\mathrm{Spec}\,A.$

It is well-known that $\mathrm{Hilb}_X^n$ is smooth when $X$ is a smooth and quasi-projective curve.
Thus, we obtain an analogous result in the non-commutative case under the hypothesis "$A$ coherent of global dimension one".

Let $<[A,A]>$ denote the ideal generated by the commutators. When $n=1$ we have that $\ran=\Hilb_A^1 \cong \mathrm{Spec}\,A/<[A,A]>$. 
As a consequence of the previous results these affine schemes are all smooth when $A$ is hereditary and coherent.
It seems to be reasonable that any definition of smoothness for an associative algebra should imply the smoothness of $\mathrm{Spec}\,A/<[A,A]>$ as coherence (and formal smoothness) does for hereditary algebras.

In a forthcoming paper, we will give an account of the behavior of coherent algebras $A$ of global dimension two, but carrying the extra condition that $A$ is homologically smooth. This property can be stated by saying that $A$ admits a finite projective resolution by finitely generated $A-$bimodule as an $A-$bimodule, see \cite{G1,Ko-S} for this notion.

\smallskip
The paper goes as follows. In Section \ref{ran} we recall the definition of $\ran$ as the scheme parameterizing the $n-$dimensional representations of $A.\,$  We recall also the natural action of $GL_{n}$ on $\ran$  and on $\rana . \,$

In Section \ref{globdim} we recall the notions of global dimension of an algebra. Section \ref{smalg} is devoted to introduce formal smoothness and coherent algebras. Then we mention the known results on the smoothness of $\Hilban$ and $\ran \,$ for formally smooth algebras.

\smallskip
In Section \ref{defsmooth} we prove our result. It is exactly at this point the the coherence of $A$ play its role. To prove the smoothness of a $k$-point in $\ran$, indeed, we use the resolution of $A$ by finite rank free $A$-bimodules  to compute the Hochschild cohomology of $A$ with coefficients in $End_k(k^n)$. We are then able show that, for every rational point in $\ran$, there is an open neighborhood of it in which the dimension of the tangent space is constant and the formal deformations surject on the infinitesimal ones. This is obtained by using a suitably adapted argument due to Geiss and de la Pena (see \cite{Ge,Ge-P}), which worked only for finite dimensional $k-$algebras.

\smallskip
In Section \ref{defnori} we first recall the definition of $\Hilban .\,$ Then we define an open subscheme  $ \uan $ in $ \rana \,$ and we prove that $\Hilban$ is smooth if and only if $\uan$ is smooth (see \ref{uasm}).
In Theorem\ref{herhilb} we prove the smoothness of $\uan$ when $A$ is hereditary and coherent.

Finally, in \ref{h1}  we show that $\Hilb_A^1 \cong \mathrm{Spec}\,A/<[A,A]>$ is smooth when $A$ is hereditary and coherent.

\section{Notations}\label{not} Unless otherwise stated we adopt the following
notations:
\begin{itemize}
\item $k$ is an algebraically closed field.
\item $B$ is a commutative $k-$algebra.
\item $F=k\{x_1,\dots,x_m\}$ denotes the associative free $k-$algebra on
$m$ letters.
\item $G=GL_{n}$ is the general linear group scheme over $k.$
\item $A\cong F/J$ is a finitely generated associative $k-$algebra.
\item $\mathcal{N}_-,\,\mathcal{C}_-$ and $Sets$ denote the
categories of $\ -$algebras,
commutative $\ -$algebras and sets, respectively.
\item The term "$A-$module" indicates a left $A-$module,
\item $-{\operatorname{mod}} , \, -{\operatorname{bimod}} $ denote
the categories of left $-$modules and $-$bimodules,
respectively.
\item we write $\Hom_{\A}(B,C)$ in a category $\A$ with
$B,C$ objects in $\A$. If $\A=A-\md$, then we will write $\Hom_{A}(-,-)$.

\end{itemize}

\subsection{The scheme of $n-$dimensional representations}\label{ran}

\n
We start recalling the definition of $\ran ,\,$ the scheme of $n-$dimensional representations of algebras.

\smallskip
\n
Denote by $M_n(B)$ the full ring of $n \times n$ matrices over
$B,\,$ with $B$ a ring. If $f \ : \ B \to C $ is a ring homomorphism we
denote with
$$
M_n(f) :  M_n(B) \to M_n(C)
$$
the homomorphism induced on matrices.

\begin{definition}
Let  $A \in \mathcal {N}_k,\, B \in \mathcal {C}_k.\,$ By an {\em n-dimensional representation of} $A$ over $B$ we mean a homomorphism of $k-$algebras $
\rho\, : A \, \to M_n(B).\,
$
\end{definition}

\smallskip
\n
It is clear that this is equivalent to give an $A-$module structure on $B^n.$
The assignment $B\to \Hom_{\nn_k}(A,M_n(B))$ defines a covariant functor
\[
\mathcal {C}_{k} \longrightarrow Sets .
\]
This functor is represented by a commutative $k-$algebra. More precisely, there is the following
\begin{lemma}\cite[Lemma 1.2.]{DPRR}\label{rep}
For all $A\in\nn_k$ and $\rho\, : A \, \to M_n(B)$ a linear representation, there exist a commutative $k-$algebra $V_n(A)$ and a
representation $\pi_A:A\to M_n(V_n(A))$ such that $\rho\mapsto
M_n(\rho)\cdot\pi_A$ gives an isomorphism
\begin{equation}\label{proc}
\Hom_{\C_k}(V_n(A),B)\xrightarrow{\cong}\Hom_{\nn_k}(A,M_n(B))\end{equation}
for all $B\in C_k$.
\end{lemma}

\smallskip
\n
\begin{proof}
Consider at first the case $A=F.\,$ Define  $V_n(F):=k[\xi_{lij}]$ the polynomial ring in variables $\{\xi_{lij}\,:\, i,j=1,\dots,n,\, l=1,\dots,m\}$ over $k.\,$ To any $n-$dimensional representation $\rho :  F \rightarrow M_n(B)$ it corresponds a unique $m-$tuple of $n\times n$ matrices, namely the images of $x_1,\dots,x_m$, hence a unique $\bar{\rho}\in \Hom_{\C_k}(k[\xi_{lij}],B)$ such that $\bar{\rho}(\xi_{lij})=(\rho(x_l))_{ij}$.
Following Procesi {\cite{Pr,DPRR}} we introduce the generic matrices.
Let $\xi_l=(\xi_{lij})$ be the $n\times n$ matrix whose $(i,j)$ entry is $\xi_{lij}$ for $i,j=1,\dots,n$ and $l=1,\dots,m$.  We call   $\xi_{1},\dots,\xi_{m}$ the {\it generic $n\times n$ matrices}.
Consider the map
\[
\pi_F:F \to
M_n(V_n(F)),\;\;\;\;x_{l}\longmapsto \xi_{l},\;\;\; l=1,\dots,m\,.
\]

\smallskip
\n
It is then clear that the map
\[
\Hom_{\C_k}(V_n(F),B)\ni \sigma\mapsto M_n(\sigma)\cdot\pi_F\in \Hom_{\nn_k}(F,M_n(B))
\]

\smallskip
\n
gives the isomorphism (\ref{proc}) in this case.

\smallskip
\n
Let now $A= F/J$ be a finitely generated $k-$algebra and $\bt:F\to A.\,$
Write $a_{l}=\beta(x_{l})$, for $l=1,\dots,m .\,$
An $n-$dimensional representation
\[
\rho : A \rightarrow M_n(B)
\]
lifts to one of $F$ by composition with $\beta .\,$ This gives a homomorphism
\[
V_n(F)=k[\xi_{lij}]\to B
\]
that factors through the quotient $k[\xi_{lij}]/I,\,$ where
$I$ is the ideal of $V_n(F)$ generated by the $n \times n $ entries of $f(\xi_1, . . . ,\xi_m)$,
where $f$ runs over the elements of $J.\,$

\n
Define $V_n(A)=k[\xi_{lij}]/I$ and $\xi_l^A=(\xi_{lij}+I)=\xi_l+M_n(I)\in M_n(V_n(A))$ for $l=1,\dots,m$. There is then a homomorphism
\[\pi_A:A\to M_n(V_n(A))\]
given by $\pi_A(a_l)=\xi_l^A$ for $l=1,\dots,m$.
To conclude, given $\rho\in \Hom _{\nn_k}(A,M_n(B))$ there is a unique homomorphism of commutative $k-$algebras
\begin{equation}\label{class}
\bar{\rho} \ : V_n(A) \to B
\end{equation}
for which the following diagram
\begin{equation}\label{univ1}
\xymatrix{
A \ar[r]^(.3){\pi _A} \ar[rd]_{\rho}&
M_n(V_n(A))\ar[d]^{M_n(\bar{\rho})} \\
& M_n(B) }
\end{equation}
commutes.
\end{proof}

\smallskip
\noindent

\begin{definition}\label{pigreco}
We write $\ran$ to denote Spec\,$V_n(A).\,$ It is considered as a
$k-$scheme.  The map
\begin{equation}
\pi _A :A \to
M_n(V_n(A)),\;\;\;\;a_{l}\longmapsto \xi_{l}^A.
\end{equation}
is called {\em the universal n-dimensional representation.}
\end{definition}

\smallskip
\noindent

\begin{example}\label{commuting}
Here are some examples.
\end{example}

\n
(1) If $A=F,\,$ then $\mathrm{Rep}_F^n (k) = M_n(k)^m.\, $

\smallskip
\n
(2) If $A=F/J,\,$ the $B-$points of  $\ran$ can be described in the following way:
\[
\ran(B)= \{(X_1,\dots,X_m) \in M_n(B)^m \ : \ f(X_1,\dots, X_m)=0\; \mbox{for all} \, f \in J \}.
\]

\n
(3) If $A=k[x],\,$ then $\ran (k)= M_n(k).\,$

\n
(4) If $A=\CC [x,y],\,$ then
$$
\ran (\CC) \ = \ \{M_1,\, M_2 \in M_2(\CC) \ : \ M_1M_2 = M_2 M_1 \}
$$

\n
is the \textit{commuting scheme}.

\bigskip
\n
Being $A$ finitely generated, $\ran$ is of finite type. Note that $\ran$ may be quite complicated. It is not reduced in general and it seems to be hopeless to describe the coordinate ring of its reduced structure. However, there is an easy description of the tangent space to $\ran$ at the $k-$points. Let us recall the following

\smallskip
\n
\begin{definition}
The space of $k$-linear $\rho$-derivations from $A$ to $M_n(k)$, denoted as
$Der _{\rho}(A,M_n(k))$ , is the space of  $k$-linear maps
\[
\varphi : A \longrightarrow M_n(k) \;\;\;\; \mbox{such that}\;\;\; \varphi(ab)=\rho(a)\varphi(b)+\varphi(a)\rho(b).\,
\]
\end{definition}

\smallskip
\n
A $k-$point $\rho: A \rightarrow M_n(k)$ in $\ran$ induces an $A-$module structure on $k^n\,$ and a $\rho-$derivation is a $k-$linear map of this module.

\smallskip
\n
\begin{lemma}(\cite[12.4]{G1})\label{tgran}
Let $\rho$ be a $k-$ in $\ran.\,$ The tangent space $T_{\rho}\ran$ is isomorphic to $Der _{\rho}(A,M_n(k)).$
\end{lemma}

\begin{proof}
Let $\rho $ be a $k$-point in $ \ran .\,$ A point in the tangent space $T_{\rho}\ran \,$ is given by a $k$-linear map $\pi : A \longrightarrow M_n(k)$ such that
\[
\rho + \varepsilon \pi : A \longrightarrow M_n(k[\varepsilon]/\varepsilon ^2)
\]
is a ring morphism. This gives that $\pi (ab)=\rho (a)\pi(b)+ \pi(a)\rho(b),\,$ hence
 $\pi \in Der _{\rho}(A,M_n(k)).\, $ Moving backwards we obtain $ T_{\rho}\ran \cong Der _{\rho}(A,M_n(k)).\,$
\end{proof}

\medskip
\n
\subsection{Actions}\label{act}

\n
Let now  $G$ be the general linear group scheme over $k\,$ whose $B-$points form the group $ \GL_{n}(B)$ of $n \times n$ invertible matrices with entries in
$B.\,$ Let $\mathbb{A}^n_k$ the $n-$dimensional affine scheme over $k.\,$
\begin{definition}\label{actions}
 Given $ B \in  \mathcal{C}_k,\,g \in G(B),\, \rho \in \ran(B) $ and $a \in A ,\,$ define
\[
\begin{matrix}
G(B) \times \ran(B) &  \longrightarrow & \ran(B) \\
(g,\rho) & \longrightarrow & \rho ^g \ : \ \rho^g(a)= g\rho(a)g^{-1} .\, \\
\end{matrix}
\]

\smallskip
\n
Analogously, for $v \in \mathbb{A}^n_k(B)$ define

\[
\begin{matrix}
G(B) \times \ran(B) \times _k \mathbb{A}^n_k(B) &  \longrightarrow & \ran (B) \times _k \mathbb{A}^n_k(B) \\
(g,\rho,v) & \longrightarrow & ( \rho ^g,gv).
\end{matrix}
\]
\end{definition}

\smallskip
\n
\begin{remark}\label{isorho}
The $A-$module structures induced  on $B^n$ by two representations $\rho$ and $\rho '$ are
isomorphic if and only if there exists $g \in G(B)$ such that $\rho
' = \rho ^g .$
\end{remark}

\begin{definition}\label{modrep}
We denote by $\ran// G=\mathrm{Spec}\,V_n(A)^{G(k)}\,$ the
categorical quotient (in the category of $k-$schemes) of $\ran$ by
$G.\,$ It is the \textit{(coarse) moduli space of $n-$dimensional
linear representations} of $A$.
\end{definition}

\section{Global dimension}\label{globdim}

\n
Let $M$ be an $A-$module. A {\em projective} resolution is an acyclic complex
\[
\dots \rightarrow P_{3} \rightarrow P_2 \rightarrow P_1 \rightarrow P_0 \rightarrow M \rightarrow 0
\]
where $P_i$ are projective $A-$modules.

\begin{definition}
Let $M$ be an $A-$module. The \textit{(projective) dimension} $pd(M)$ of $M$ is the minimum length of a projective resolution of $M.\, $
If every projective resolution of $M$ is infinite then its projective dimension is infinite.
\end{definition}

\n
For $i=0,1,...$ and $M,N\in A-$mod we denote as usual by $\Ext_{A}^i(M,N)$ the $i-$th ext-group, which is the value on $(M,N)$ of the $i-$th derived functor of $\Hom_{A}(-,-)$. See \cite{Co2} for further readings.

\n
It is possible to express the projective dimension of an $A-$module $M$ in terms of the $\Ext_{A}^i .$

\begin{proposition}(\cite[Proposition 2.6.2]{Co2})\label{prext}
The following conditions are equivalent:
\begin{enumerate}
\item Projective dimension of $M$ is $\leq n$ .
\item $\Ext ^i _{A} (M,-)=0\,$ for all $i > n .$
\item $\Ext ^{n+1} _{A} (M,-)=0 .\,$
\end{enumerate}
\end{proposition}

\smallskip
\n
\begin{definition}The \textit{global dimension} of a ring $A,\,$ denoted with $gd(A)$ is defined to be the supremum of the set of projective dimensions of all (left) $A-$modules.
\end{definition}

\smallskip
\n
The global dimension of $A$ is a non-negative integer or infinity and it is a homological invariant of the ring.
If $A$ is non commutative, note that the left global dimension can be different from the right global dimension, formed from right $A$-modules. These two numbers coincide if $A$ is Noetherian (see \cite [Cor.2.6.7.]{Co2}).

\smallskip
\noindent
\begin{remark}
It is customary to use the notation $gd(A)\leq n$ instead of $gd(A)=n,\,$ being $gd(A)$ a supremum of a set.
\end{remark}

\smallskip
\noindent
\subsection{Examples}\label{ex1}

\n
\subsubsection{Polynomial rings}
Let $A = k[x_1,...,x_n]$ be the ring of polynomials in $n$ variables over $k.$ The global dimension of $A$ is equal to $n$.  More generally, if $R$ is a Noetherian ring of finite global dimension $n$ and $A = R[x]$ is a ring of polynomials in one variable over $R,\,$ then the global dimension of $A$ is equal to $n + 1$.

\subsubsection{Commutative ring with finite global dimension}\label{serre}

\n
In commutative algebra a commutative and noetherian ring $A$ is said to be {\em regular} if the localization at every prime ideal is a regular local ring. There is the following result due to Serre

\begin{theorem}(\cite  [Chap.IV.D.]{Se})
Let $(R,m)$ a commutative noetherian local ring. Then $gd(R) < \infty$ iff $R$ is a regular local ring. In this case, $gd(R)= \dim R,\,$ where $\dim R$ is the Krull dimension of $R$.
\end{theorem}

\n
Moreover,

\begin{theorem}(\cite[5.94]{La})
If $R$ is a commutative noetherian ring, then the following are equivalent,
\begin{enumerate}
\item
$R_p$ is a local regular ring for any $p \in Spec R.$
\item
$R_m$ is a local regular ring for any  maximal ideal $m$ in $R$.
\item
$pd(p)< \infty$ for any $p \in \mathrm{Spec}R.$
\item
$pd(m)< \infty$ for any maximal ideal $m$ in $R$.
\item
$pd(M)< \infty$ for any finitely generated $R-$module $M.$
\end{enumerate}
\end{theorem}

\n
If any of these conditions holds, the ring $R$ is said to be {\em regular}. For those rings, $gd(R)=\dim R.$

\begin{theorem}(\cite[5.95]{La})
Let $R$ be a commutative noetherian ring. For any $m$ maximal ideals, $gd(R_m)=pd(R/m)$ and
\[
gd(R) = sup \{gd(R_m)\}
\]
where $m$ ranges over all maximal ideals $m$ in $R.$
\end{theorem}

\n
These results imply the following

\begin{corollary}\label{fingdreg}
Let $R$ be a commutative noetherian ring. If $gd(R)< \infty,\,$ then $R$ is regular.
\end{corollary}

\n

\subsubsection{Dimension zero}The rings of global dimension zero are the semisimple ones, since for these rings any module is projective (see \cite[Theorem 5.2.7.]{Co1}) The free algebra $F$ has dimension zero.

\subsubsection{Dimension one - Hereditary algebras}\label{hereditary} The global dimension of a ring $A$ is less or equal than $1$ if and only if all (left) modules have projective resolutions of length at most $1$. These rings are called {\em hereditary}.
Hereditary rings are for example semisimple rings, principal ideal domains, ring of upper triangular matrices over a division ring. In the commutative case, a hereditary domain is precisely a Dedekind domain (see \cite[Section 10.5]{Co1}). Thus, the coordinate ring of an affine smooth curve is hereditary.

\subsubsection{Path algebras}
A quiver $Q$ is a finite directed graph, eventually with multiple edges and self-loops. A path $p$ is a sequence $a_1a_2...a_m$ of arrows such that $ta_i=ha_{i+1}$ for $i=1,\dots m-1,\,$ where $ta$ and $ha$ denote the tail and the end of an arrow $a.\,$ For every vertex $x$ in the quiver we define the trivial path $i_x.\,$ The path alegbra $kQ$ of the quiver is the vector space spanned by all paths in the quiver. Multiplication is defined by
\[
p \cdot q  \, := \,
\begin{cases}
pq &\mbox{if} \;\; tp = hq \\
0 & \mbox{otherwise}
\end{cases}
\]
for two paths $p,\,q.\,$ It is easy to see that the path alegbra of a quiver is finite dimensional if and only if the quiver does not contain self-loops and oriented cycles, i.e. paths $p$ such that $t_p=h_p .\,$
There exists a standard resolution of length one for modules over a path algebra (see \cite{C-B2}). Therefore these algebras are hereditary.

\smallskip
\n

\section{Smooth algebras}\label{smalg}

\n
The notion of "smoothness" for an algebra $A\in \mathcal{N}_k $ is not well-established, yet.
We recall here different kinds of definitions which could provide a generalization for this notion in the non-commutative setting. We address to \cite{G1,G2,Ko} for further readings.

 We quote the known results on the smoothness of $\ran$ given by the "smoothness" of $A.\,$ Finally, we present a not known result for $A$ of global dimension one using the concept of "coherent algebra".

\subsection{Formally smooth algebras}
The definition of formally smooth (or quasi-free) algebras goes back to J. Cuntz and D. Quillen and it is intended to give a generalization for the notion of free associative algebra.

\begin{definition}(Definition 3.3. \cite{C-Q,Sc}). \label{fsm}
An associative algebra $A \in \mathcal N _k$ is said to be \textit{formally smooth} (or  \textit {quasi-free}) if it satisfies the equivalent conditions:

\smallskip
\n
i)
any homomorphism $ \varphi \in \Hom_{\mathcal N _k}(A,R/N)$ where $N$ is a nilpotent (bilateral) ideal in an algebra $R \in \mathcal N _k,\,$ can be lifted to a homomorphism $\overline{\varphi} \in \Hom_{\mathcal N _k}(A,R)\,$ that commutes with the projection
$R \rightarrow R/N .$

\n
ii) $\Ext^2_{A-bimod}(A,M)=0$ for any $A-$bimodule $M.$

\n
iii) the kernel $\Omega ^1 _A$ of the  multiplication $A \otimes A \rightarrow A $ is projective in $A-bimod .\,$
\end{definition}

\begin{remark}
It is worth to stress here that if we consider $A \in Ob(\mathcal C _k)$ and $ \Hom_{\mathcal C _k}(A,-)$ in the above, we obtain the classical definition of regularity in the commutative case (see \cite[Proposition 4.1.]{LB}).
On the other hand, if we
ask for a commutative algebra $A$  to be formally smooth in the category $\mathcal N _k$ we obtain regular algebras of dimension $\leq 1\,$ only (\cite[Proposition 5.1.]{C-Q}).
\end{remark}

\smallskip
\n
\begin{remark}
When $A$ is commutative $\Omega ^1 _A$ is nothing but the module of the Kahler differentials (see \cite[Section 8]{G1}).
\end{remark}

\smallskip
\noindent
\begin{example} Here are a few examples of formally smooth algebras (see \cite[19.2]{G1}).
\begin{enumerate}
\item The free algebra $F.$
\item The matrix algebra $M_n(k).$
\item The path algebra of a quiver.
\item The function ring $k[X],\,$ where $X$ is a smooth affine curve.
\end{enumerate}
\end{example}

\smallskip
\noindent
Actually, formally smooth algebras form a very restricted class, as we see in what follows.

\smallskip
\noindent
\begin{lemma}(\cite[Lemma 19.1.6.]{G1}\label{cq})
A formally smooth algebra is hereditary.
\end{lemma}

\begin{remark} From Lemma \ref{cq} follows that the polynomial algebra $k[x_1,\dots,x_n]$ is not formally smooth for $n > 1.\,$
In general, the tensor product of two formally smooth algebras is not formally smooth.
\end{remark}

\begin{proposition}(\cite{C-Q})\label{herform}
If $k=\mathbb C$ and $A$ is a finite-dimensional hereditary algebra, then $A$ is formally smooth.
\end{proposition}

\smallskip
\n
In the general case the situation is more complicated, as can be seen in the following

\begin{proposition}\cite[Proposition 8.5.]{C-B3}
A finite-dimensional algebra $A$ is formally smooth iff it is hereditary and $A/rad A$ is separable over $k.$
\end{proposition}

\smallskip
\n
The formally smoothness of $A$ implies the smoothness of $\ran.$

\begin{theorem}\label{fsmran} (\cite[Proposition 19.1.4.]{G1}, \cite[Prop.6.3.]{LB}) If $A \in \nn _k$ is a formally smooth $k$-algebra, then $\ran$ is smooth for every $n$.
\end{theorem}

\smallskip
\n
{\em Idea of proof.} If $I \subset B$ is a nilpotent ideal, then so is $M_n(I) \subset M_n(B). \,$ Using this observation, the formally smoothness of $A$ is used to show that the coordinate ring $k[\ran]$ verifies the lifting property i) of Definition \ref{fsm}.\qed

\medskip
\n
If  $A$ is finite dimensional, there is the following
\begin{theorem} (\cite[Proposition 1]{Bo})\label{finhersm}.
If $A$ is finite-dimensional, the scheme $\ran$ is smooth for any $n$ if and only if $A$ is hereditary.
\end{theorem}

\smallskip
\n
This result is proved directly in \cite{Bo}. Alternatively, one can use Proposition\ref{herform} and Theorem\ref{fsmran}.

\smallskip
\n
\subsection{Coherent algebras}\label{cohalg}

\n
Let $A^{\op}$ the opposite algebra of $A$ and denote with $A^e:=A \otimes A^{\op}.\,$ There is a canonical isomorphism $(A^e)^{\op} \cong A^e,\,$ so an $A-$bimodule is the same as a left $A^e$-module.

\smallskip
\n
\begin{definition}(see \cite{G2})
An algebra $A$ is said to be {\it coherent} if the kernel of any morphism between finite
rank free $A$-bimodules is finitely generated.
\end{definition}

\n
There is the following

\begin{lemma}(see \cite[9.1.1]{G2})\label{freeres}
A coherent algebra $A$ admits a resolution by finite rank free $A$-bimodules.
\end{lemma}

\begin{proof}
\n
Consider the kernel $\Omega ^1 _A$ of the  multiplication $m:A \otimes A \rightarrow A^{op} .\,$ It  is a finitely generated $A^e$-module. Therefore there exists a surjective $A^e$-module homomorphis
$
F_0 \twoheadrightarrow \Omega ^1 _A \,
$
where $F_0$ is a free $A^e$-module of finite rank. Being $A$ coherent, the kernel of the map
\[
F_0 \longrightarrow A \otimes A^{op}
\]
will be finitely generated. Thus there exists $F_1$ free $A^e-$module of finite rank and an exact sequence $F_1 \rightarrow F_0 \rightarrow A \otimes A  \rightarrow A .\,$ Continuing in this way, we
obtain an exact sequence
\[
\dots \longrightarrow F_i \longrightarrow \dots \longrightarrow F_1 \longrightarrow F_0 \longrightarrow A \otimes A  \longrightarrow A \longrightarrow 0
\]
where all the $F_i$'s are finite rank free $A-$bimodules.
\end{proof}

\smallskip
\n
\begin{remark}\label{comcoh}
The class of formally smooth algebras and of coherent algebras are disjoint. For example, a formally smooth algebra must have global dimension $\leq 1$, see (Lemma \ref{cq}).
\end{remark}

\section{Deformations and smoothness}\label{defsmooth}

\n
Our aim here is to prove the smoothness of $\ran$ when $A$ is coherent and of global dimension one (hereditary).
We have seen in Prop-\ref{herform} that if $A$ is finite-dimensional, then it is hereditary if and only if it is formally smooth. In the infinite dimensional case this is no more true. We prove that $\ran$ is smooth if $A$ is coherent and of global dimension one.

\smallskip
\n
Let $\mu$ be a $k-$point in $\ran$ and let $M\cong k^n$ the associated $A-$module (see Example \ref{commuting} (1)). We'll prove that $\mu$ is smooth if the following condition holds:

\medskip
\n
{\it(*) The algebra $A$  is coherent and $\Ext^2_{A}(M,M)=0.$}

\medskip
\n
This rephrases a result of Geiss (see \cite[6.4.2.]{Ge}) which holds for finite dimensional algebras.

\smallskip
\n
Denote by $\cdot$ the multiplication on $A$ and by an element $\mu \in \Hom_k(A \otimes k^n , k^n )\, $ the multiplication on the $A-$module $M=k^n.\,$  If $R$ is a local commutative $k-$algebra, denote $A_R:=A \otimes R$ and $M_R:= M \otimes R.\,$ Define the deformations of $M$ in the following way

\begin{definition}(\cite[3.2.]{Ge})
A $R-${\it
deformation} of an $A-$module $M$ is given by an $R-$linear map  $\mu _R \in \Hom_R (A_R \otimes M_R,M_R)$ which reduces modulo $R$ to $\mu$ and satisfies
$$
\mu _R(id_{A_R}\otimes _R \mu _R)=\mu _R (\cdot \otimes id_{M_R}).
$$

\n
If $R=k[[t]],\,$ the ring of formal power series, we speak of {\it formal deformations}, if $R=k[\epsilon]:=k[t]/(t^2),\,$ the algebra of dual numbers, we speak of {\it infinitesimal deformations}.
\end{definition}

\n
\begin{definition}
An infinitesimal deformation $\mu_{k[\epsilon]}$ is called {\it integrable} if there exists a deformation $\mu _{k[[t]]}$ of $\mu$ which reduces to $\mu_{k[\epsilon]}$ via the natural projection $p_{t,\epsilon}:  k[[t]] \rightarrow k[\epsilon]$.
\end{definition}

\smallskip
\n
We prove our result using the following criterium for a point $x$ in a scheme $X$ to be smooth.

\n
\begin{lemma}(\cite[4.4.3.]{Ge})\label{crit}
Let $X$ be an algebraic scheme over $k$ and suppose that $x \in X(k)$
has an open neighborhood $U$ such that for all $y \in U(x)$ the following conditions hold:

\n
(i) For each $y' \in T_{U,y}$ we have $(U\cap(p_{t,\epsilon})^{-1}(y')) \neq \emptyset ,\, $  where $p_{t,\epsilon} : k[[t]] \mapsto k[\epsilon]$ is the canonical projection.

\n
(ii) $\dim T_{U,y} = \dim T_{U,x}.$

\n
Then $x$ is a smooth point of $X$.
\end{lemma}

\n
If $M$ is an $A-$module then $\End _k (M)\,$ is an $A-$bimodule. Consider the Hochschild cohomology of $A$ with coefficients in $\End_k(M).\,$ There is the following standard result

\begin{theorem}(\cite[Corollary 4.4.]{C-E})
For all $A-$module $M$
\[
\Ext ^i_{A}(M,M)\cong H^i_{A^e} (A, \End _k (M)).
\]
\end{theorem}

\smallskip
\n
Now we are able to prove the following

\begin{theorem}\label{ness}
If $A$ is a coherent algebra and  $M\cong k^n$ is an $A-$module with $Ext^2(M,M)=0,\,$ then $M$ is a smooth point in $\ran.$
\end{theorem}

\begin{proof}
A point $\mu$ in $ \ran(k) $  corresponds to an $A-$module  $M\cong k^n.\,$
It is well-known that the obstructions to extend an infinitesimal deformation of the module  $M$ to a formal one are in $\Ext^2_{A}(M,M),\,$ see for example \cite[3.6. and 3.6.1.]{Ge}, so the first condition of Lemma \ref{crit} is satisfied by hypothesis.

\n
Being $A$ coherent, it admits a resolution by finite rank free A-bimodules (see Lemma \ref{freeres}):
\[
\dots \longrightarrow F_{s} \longrightarrow \dots \longrightarrow F_1 \longrightarrow F_0 \longrightarrow A^e \longrightarrow A \longrightarrow 0.
\]
\smallskip
\noindent
Apply $\Hom _{A^e} (-,\End _k(M))$ to the previous resolution:
\begin{equation}\label{longi}
\begin{array}{ll}
0 & \longrightarrow \Hom _{A^e} (A,\End _k(M)) \longrightarrow \Hom _{A^e} (A^e,\End _k(M)) \stackrel{d_{\mu}^0}{\longrightarrow}
\\
\\
&\stackrel{d_{\mu}^0}{\longrightarrow} \Hom _{A^e} (F_0,\End _k(M))
\stackrel{d_{\mu}^1}{\longrightarrow} \Hom _{A^e} (F_1,\End _k(M)) \stackrel{d_{\mu}^2}{\longrightarrow} \dots
\end{array}
\end{equation}

\medskip
\n
All the spaces $\Hom _{A^e} (F_i,\End _k(M))\cong M_n(k)^{rank(F_i)}$ are finite dimensional over $k.\, $  Furthermore, their ranks do not depend on $\mu .\,$ In particular note that $\Hom _{A^e} (A^e,\End _k(M))\cong \End _k(M) \cong M_n(k).\, $
The cohomology groups of the sequence \eqref{longi} are the Ext groups
\[
\Ext ^i_{A}(M,M)\cong H^i_{A^e} (A, \End _k (M))
\]
and by hypothesis $\Ext^2_A(M,M)=0.\,$
Therefore $ \dim \ker d^1_{\mu}=\dim \im \ d ^0 _{\mu}$ and
\begin{equation}\label{k01}
\dim \ker d ^0 _{\mu} = -\dim \ker d^1 _{\mu} + \dim \Hom _{A^e}(A^e, \End _k (M)).
\end{equation}
Observe now that the functions
\[
z^i \ : \ \ran (k) \longrightarrow  \NN , \;\;\;\;\; \mu \longmapsto  \dim_k \ker d^i _{\mu}
\]
are upper semicontinuous. Thus the function $k^0:=\dim \ker d^0 _{\mu}$ is constant in $\ran(k)\,$ by (\ref{k01}). It is well known that (see for example \cite[5.4.]{G1})
\[
\ker d^0 _ {\mu} \cong Der _{\mu}(A,M_n(k)).
\]
By Lemma \ref{tgran} we conclude that $T_{\mu}\ran \cong \ker d^0_{\mu}.\,$ It follows that
the dimension of the tangent space  $T_{\mu}\ran$ is constant. This gives $(ii)$ of Lemma \ref{crit} and we can conclude that $\mu$ is smooth in $\ran .$
\end{proof}

\section{Smoothness in dimension one}

\n
The aim here is to apply Theorem \ref{ness} to show that $\ran$ is smooth for hereditary algebras which are coherent.

\smallskip
\n

\begin{theorem}\label{e20}
Let $A$ be a coherent and hereditary algebra, then $\ran$ is smooth.
\end{theorem}

\begin{proof}
Let $A$ be an associative algebra with $gd(A)\leq 1\,$ and let $M$ be an $A-$module, then $\Ext^2_{A}(M,M)=0 $
(see Proposition \ref{prext}).
\end{proof}

\smallskip
\n

\section{Nori - Hilbert schemes}\label{defnori}

\n
In this paper the Nori - Hilbert scheme is the representing scheme of a functor of points $\mathcal{C}_k \to Sets ,\,$ which is given on objects by
\[
\begin{array}{ll}
\Hilb_A^n(B) := & \{\mbox {left ideals } I
\mbox { in } A \otimes_{k} B \mbox { such that } M=(A \otimes _k B) /I \
\mbox {is projective} \\
& \mbox{ of rank } n
\mbox { as a } B\mbox{-module} \}.
\end{array}
\]
where $A \in \mathcal{N}_k,\,\,B\in \mathcal{C}_k.\,$ Nori introduced this functor in \cite{No} only for the case $A=\mathbb Z \{x_1,\dots,x_m\}.\,$
$\Hilban$ is a closed subfunctor of the Grassmannian functor, so it is representable by a scheme (see \cite {VdB},
Proposition 2 ). It is also called {\textit{the non commutative Hilbert scheme}} (see \cite{Re}) or the {\textit {Brauer-Severi scheme}} of $A$ (see \cite {VdB,LB}).

\smallskip
\noindent
It was used to propose a desingularization of  $\ran // G .\,$ The candidate for the desingularization was a subscheme $V$ in $\Hilban \,$ and Nori showed that $V$ gives a desingularization for $n=2,\,$ but it was proved recently in \cite{Ol} that $V$ is singular for $n \geq 3.\,$

The scheme $\Hilban$ was then defined in a more general setting in \cite{VdB} and in \cite{Re}. It was also referred as {\it{Brauer-Severi scheme}} of $A\,$in \cite{LB,Le,Seel}  in analogy with the classical Brauer-Severi varieties parameterizing left ideals of codimension $n$ of central simple algebras (see \cite{A}). It is also called {\textit{the non commutative Hilbert scheme}} (see \cite{G-V,Re}). A generalization of the Hilbert to Chow morphism has been studied in \cite{G-V} by using the results of the second author \cite{V3, V1, V}. We address the interested reader to {\cite{V}} for a survey on these topics.

\smallskip
Van den Bergh showed that for $A=F$ the free associative algebra on $m$ variables, $\mathrm{Hilb}^n _F$  is smooth of dimension $n^2(m-1) +n ,\,$ (see \cite{VdB}).

\subsection{Hilbert schemes of $n$-points.}\label{npoints}
\label{commcase} Let now $A$ be commutative and  $X=\mathrm{Spec} A .\,$ The $k-$points of $\Hilb_A^n$ parameterize zero-dimensional
subschemes $Y\subset X$ of length $n.\,$ It is the simplest case of Hilbert scheme
parameterizing closed subschemes of $X$ with fixed Hilbert polynomial $P,\,$  in this case $P$ is the constant polynomial $n.\,$ The scheme $\Hilb_A^n$ is
usually called the {\em Hilbert scheme of $n-$points on} $X \,$ (see for example Chapter 7 in \cite{BK,Ia} and Chapter 1 in \cite{Na}).

\n
The Hilbert scheme of $n$-points of a smooth quasi-projective curve or surface is smooth.

\n
\begin{theorem} (see  \cite{De,Iv,Fo})\label{sthilb}
If $X$ is an irreducible an smooth curve,  then the Hilbert scheme of the $n-$points over $X$ is non singular and irreducible of dimension $n.\,$
\end{theorem}

\smallskip
\n
These important result is based on the connectedness of the Hilbert scheme  and on the existence of an open non singular subset of known dimension $n.\,$ The Serre Duality Theorem is the key tool to prove that the tangent space has constant dimension $n.\,$

\medskip
\n
\subsection{An open subscheme in $\ran$}\label{uan}

For any $B\in\C_k\,$ identify $B^n$ with $\mathbb{A}^n_k(B).\,$

\begin{definition}\label{defuan}
For each $B\in \C_k$, consider the set
\[
\uan(B) \ = \ \{(\rho,v) \in \rana \ : \
\rho(A)(Bv)=B^n\}.
\]
The assignment $B\mapsto \uan(B)$ is functorial in $B$, so that
we get a subfunctor $\uan$ of  $\rana$ that is clearly open. We denote by $\uan$ the open subscheme of $\rana$ which represents it.
\end{definition}

\smallskip
\n
\begin{remark}\label{cyclic}
Note that the $A-$modules corresponding to points in $\uan$ are the cyclic modules.
\end{remark}

\smallskip
\n
Nori proved the following when $A=F.$

\begin{theorem}(\cite[Proposition 1]{No})\label{norif} The geometric quotient $\ufn / G$ exists and the natural map $\pi _F : \ufn \longrightarrow \ufn / G $ is a principal $G-$bundle. Moreover, $\ufn / G$ represents $\Hilbfn .$
\end{theorem}

\smallskip
\n
This results has been generalized to the case $A=R\{x_1,\dots,x_m\}$ for $R$ commutative ring in \cite[Theorem 7.16]{Ba}. See also Proposition 6.3 in \cite{LB} which states the same result for the $\CC -$points in $\uan$.

\smallskip
\noindent
\begin{lemma}\label{pairs}
There exists a morphism $ u_A^n  \, : \, \uan \longrightarrow \Hilban $ which is surjective on the $k-$points.
\end{lemma}

\begin{proof}
Take $(\rho,v) \in \uan (B)\, $ and consider the surjective map
\[
A \otimes B  \longrightarrow B^n , \;\;\;\;\; a \otimes b \longmapsto \rho(a) b v,
\]
its kernel $I$ is a $B-$point in $\Hilban .\, $
The assignment $(\rho,v) \mapsto I $
is natural in $B$ giving thus a morphism  $u_A^n : \uan \longrightarrow \Hilban $.
This morphism is surjective on the $k-$points, since any projective $k-$module of rank $n$ is isomorphic to $k^n.$
\end{proof}

\smallskip
\n
Consider now the composition
\[
\uan  \stackrel {i}{\hookrightarrow} \ufn \stackrel {\pi _F}{\longrightarrow} \ufn /G
\]
of the closed embedding $i$ given by the surjection  $F \rightarrow A=F/J$  together with the quotient map $\pi _F.$
By using \cite[Theorem 3.9]{Ba} we have the following

\begin{proposition}
The schematic image $Z$ of $\uan$ under the composition $\pi _F \cdot i\,$ is the descent of $\uan$ under $\pi _F ,\,$ that is,
\[
\uan \cong Z \times _{\ufn  /G } \ufn.
\]
Thus $\uan$ is identified with the Zariski locally-trivial principal $G-$bundle given by the pull back of $\pi _F$ to $Z$. In particular $Z$ is smooth if and only if $\uan$ is smooth.
\end{proposition}
\medskip
\noindent
This is the last observation we need to prove the main result of this section.
\begin{proposition}\label{schemim}
The Hilbert scheme $\Hilb ^n_A$ and the schematic image $Z$ coincide on the $k-$points.
\end{proposition}
\begin{proof}
Denote $j : F \rightarrow A=F/J\,$ and take $(\rho,v) \in \uan (B).\,$ Note that
\[
(\pi _F \cdot i) (\rho,v) \, = \, \pi _F (\rho \cdot j,v).
\]
Thus, the kernel $I$ of the map
\[
F \otimes B  \longrightarrow B^n , \;\;\;\;\; f \otimes b \longmapsto (\rho \cdot j)(f) b v
\]
is contained in $\Hilban (B),\,$ that is, $u^n_F(\rho \cdot j)=I \in \Hilban (B)$ (see Lemma \ref{pairs}).
This means that the image $(\pi_F \cdot i) (\uan)$ is contained in $\Hilban .\,$ Now, from the defining properties of the schematic image follow that $Z \subset \Hilban .\,$
We have seen that $u^n _A$ is surjective on the $k-$points and this gives the result.
\end{proof}

\smallskip
\n
\subsection{A smoothness criterion }\label{smoocrit}
From the previous two propositions we have the following.

\smallskip
\n
\begin{theorem}\label{uasm}
The scheme $\Hilban $ is smooth iff $\uan $ is smooth.
\end{theorem}
\begin{proof}
We have the following equivalences:

\smallskip
\n
$\Hilban{\text{ is smooth }} \iff {\text{ all points in }}\Hilban(k)=Z(k){\text{ are smooth }}\iff Z\cong \uan/G {\text{ is smooth }}
\iff\uan{\text{ is smooth }}.$
\end{proof}

\smallskip
\n
The smoothness of $\ran$ implies that the open $\uan$ is smooth too. Therefore Theorem \ref{uasm} implies the following

\begin{theorem}
If $A$ is formaly smooth, then $\Hilban$ is smooth.
\end{theorem}

\n
From the results of Section \ref{defsmooth} we derive

\smallskip
\noindent
\begin{theorem}\label{herhilb}
Let $A$ be a hereditary and coherent algebra, then $ \Hilban$ is smooth.
\end{theorem}

\begin{proof}
We know by Theorem\ref{uasm} that $\Hilb _A ^n$ is smooth if and only if $\uan$ is smooth. A point $ I \in \Hilban (k) $ corresponds to a point in $(\mu,v) \in \uan (k)$ and $\mu$ is smooth in $\ran$ by Theorem\ref{ness}. This means that the corresponding point is smooth in $\uan$ which is open in $\rana .\,$ Thus we have proved that $(\mu,v)$  is smooth in $\uan$ and we've done.
\end{proof}

\begin{remark}
If $A$ is commutative and noetherian, the previous result gives Theorem\ref{sthilb}, since hereditary algebras are always regular (see Corollary \ref{fingdreg}). They are precisely the coordinate rings of smooth affine curves (see  \ref{hereditary}).
\end{remark}

\subsection{$\Hilb_A^1$}\label{h1}
We want to study $\Hilb_A^1$.  We need to introduce the following.
\begin{definition}
Given a $k-$algebra $A$ we denote by $[A]$ the two-sided ideal of
$A$ generated by the commutators $[a,b]=ab-ba$ with $a,b\in A$. We
write
\[A^{ab}=A/[A]\]
and call it the abelianization of $A$.
\end{definition}

\begin{proposition}\label{abf}
For all $B\in \C_k$ there is an isomorphism $\C_k(A^{ab},B)\to \nn_k(A,B)$ by means of $\rho\mapsto\rho\cdot\mathfrak{ab}_A$. Equivalently for all $\varphi\in\nn_k(A,B)$ there is a unique $\overline{\varphi}:A^{ab}\to B$ such that the following diagram commutes
    \[
\xymatrix{
  A \ar[dr]_{\varphi} \ar[r]^{\mathfrak{ab}_A}
                & A^{ab} \ar[d]^{\overline{\varphi}}  \\
                & B             }
\]
\end{proposition}
\begin{proof}
Straightforward.
\end{proof}

\begin{proposition} There is an isomorphism
$\Hilb_A^1\cong \mathrm{Spec}\ A^{ab}$
of $k-$schemes.
\end{proposition}
\begin{proof}
As we observed, for all $B\in\C_k$ a $B-$point $I<A\otimes_k B$ of $\Hilb_A^1$ can be realized as $\ker(A\to B)$ where the map is given by $a\mapsto \rho(a)v$, with $\rho:A\to M_1(B)=B$ a linear representation and $v\in B$ a cyclic vector. $B$ is commutative and therefore $I$ is bilateral. This gives, jointly with Prop. \ref{abf}, an equivalence $\Hilb_A^1(B)\cong \Hom_{\C_k}(A^{ab},B)$, natural in $B$.
\end{proof}

As a corollary of the above Proposition and the Theorem \ref{ness} we have the following.
\begin{corollary}
For a hereditary and coherent algebra $A$ it holds that $\Hilb_A^1$ is a smooth and affine scheme.
\end{corollary}

\bigskip
\centerline{Acknowledgement}
We would like to thank Corrado De Concini for sharing his ideas with us.
We warmly thank Victor Ginzburg for pointing out a serious mistake in the first version of this paper.
A special thank goes to Sandra Di Rocco for having invited us to the KTH Department of Mathematics.
We would also like to thanks the organizers of the semester and Mittag-Leffler Institut staff.

\bibliographystyle{amsplain}

\begin{thebibliography}{99}

\bibitem{A} M. \ Artin, Brauer-Severi Varieties in \textit{Brauer Groups in Ring Theory and Algebraic Geometry}. Lect. Notes Math. \textbf{917} (1982).

\bibitem{Ba} T. E.\ Venkata \ Balaji, Limits of rank 4 Azumaya algebras and applications to desingularization. \textit{ Proc. Indian Acad. Sci. (Math. Sci.)} \textbf{112} (2002), 485-537.

\bibitem{Bo} K. \ Bongartz, A geometric version of the Morita equivalence. \textit{ J. Algebra}  \textbf{139} (1991), 159-171.

\bibitem{BK} M. \ Brion , S.\  Kumar, \textit{ Frobenius splitting methods in geometry and representation theory}. Progress in Mathematics, 231.  Birkh�user, Boston (2005).

\bibitem{C-E} H. \ Cartan and S. \ Eilenberg, \textit{ Homological Algebra}. Princeton University Press (1956).

\bibitem{Co1} P.M. \ Cohn, \textit{Basic Algebra}. Springer (2002).

\bibitem{Co2} P.M. \ Cohn, \textit{Further Algebra and Applications}. Springer (2003).

\bibitem{C-B2} W.\ Crawley-Boevey, Lectures on Representations of Quivers. Notes available at www.maths.leeds.ac.uk/~pmtwc/quivlecs.pdf.

\bibitem{C-B3} W.\ Crawley-Boevey, Preprojective algebras, differential operators and a Conze embedding for deformations of Kleinian singularities. \textit{ Comment. Math. Helv.} \textbf{74} (1999), 548-574.

\bibitem{C-Q} J. \ Cuntz   and D.\ Quillen, Algebra extension and non singularity. \textit { J.  Amer.Math.Soc.}  \textbf {8} (1995), 251-289.

\bibitem{DPRR} C.\ De Concini, C. \ Procesi ,  N.\ Reshetikhin,  M.\ Rosso, Hopf algebras with trace and representations. \textit{ Invent. Math.} \textbf {161} (2005), 1-44.

\bibitem{De} P. \ Deligne, \textit {Cohomologie \'a support propres} in Th\'eories des Topos et Cohomologie \'Etale des Schemas, Tome 3. S\'eminaire de G\'eometrie Alg\'ebrique du Bois Marie 1963/64, SGA 4. Lecture Notes in Mathematics 305 , 250-480 (1973).

\bibitem{Fo} J. \ Fogarty, Algebraic families on an algebraic surface.\textit{Amer. Jour. of Math.} \textbf{90}  (1968), 511-521.

\bibitem{G-V} F. \ Galluzzi and F. \ Vaccarino, Hilbert-Chow Morphism for Non Commutative Hilbert Schemes and Moduli Spaces of Linear Representations. \textit{Algebras and Representation Theory}. \textbf{13} (2010), 491-509.

\bibitem{Ge} C. \ Geiss, Deformation Theory of Finite Dimensional Modules and Algebras. Lectures given at ICTP (2006).

\bibitem{Ge-P} C. \ Geiss , J. A. \ de la Pena, On the Deformation Theory of Finite Dimensional Modules and Algebras. \textit{ Manuscripta Math.} \textbf{88}  (1995),  191-208.

\bibitem{G1} V.\ Ginzburg, Lectures on Non Commutative Geometry. ArXiv math.AG/0506603.

\bibitem{G2}  V. \ Ginzburg, Calabi-Yau algebras, Arxiv math.AG/0612139.

\bibitem{Ia} T. \ Iarrobino, \textit{Hilbert scheme of points: Overviews of the last ten years in Algebraic Geometry}. Bowdoin 1985, PSPM 46 Part 2, Amer. Math. Soc.,Providence, 297-320 (1987).

\bibitem{Iv} B. \ Iversen, \textit{Linear determinants with applications to the Picard Scheme of a family of algebraic curves}. Lecture Notes in Mathematics, \textbf{174}, Springer (1970).

\bibitem{Ko-R} M. \ Kontsevich , A. \ Rosenberg, Noncommutative smooth spaces. ArXiv math.AG/9812158.

\bibitem{Ko-S} M. \ Kontsevich , Y. \ Soibelman, \textit{Notes on $A_{\infty}$-algebras, $A_{\infty}$-categories and
non-commutative geometry.} Lect. Notes Phys. \textbf{757}, 153-219 (2009).

\bibitem{La} T.J.\ Lam, \textit{Lectures on Modules and Rings}. GTM \textbf{189}, Springer Verlag, Berlin (1998).

\bibitem{Le} L.\ Le Bruyn, Seelinger, G.: Fibers of generic Brauer-Severi schemes. \textit{J.\ Algebra} \textbf{214} (1982), 222-234 .

\bibitem{LB} L.\ Le Bruyn, \textit{Noncommutative Geometry and Cayley-smooth Orders}. Pure and Applied Mathematics  \textbf{290}, Chapman and Hall (2007).

\bibitem{Na} H. Nakajima, \textit{ Lectures on Hilbert schemes of points on surfaces}. University Lecture Series, 18, Amer. Math. Soc., Providence (1999).

\bibitem{No}  M.V. \ Nori, Appendix to the paper by C.\ S. \ Seshadri : \textit{Desingularisation of the moduli varieties of vector bundles over curves}. In: Proceedings of the International Symposium on Algebraic Geometry, Kyoto (1977).

\bibitem{Ol} R. \ Olbright, Nori's desingularization and its singularities, PhD Thesis (2007).

\bibitem{Pr} C. Procesi, Non-commutative affine rings. \textit{Atti Accad. Naz. Lincei Mem. Cl. Sci. Fis. Mat. Natur. Sez. I} \textbf{8} (1967), 237-255.

\bibitem{Re} M.\ Reineke, Cohomology of non-commutative Hilbert schemes. \textit{Algebr. Represent. Theory} \textbf{8} (2005), 541-561.

\bibitem{Sc} W. \ Schelter, Smooth algebras. \textit {J. of Algebra} \textbf{103} (1986), 677-685 .

\bibitem{Seel} Seelinger, G.: Brauer-Severi schemes of finitely generated algebras. \textit{Israel J. Math.} \textbf{111} (1999), 321-337.

\bibitem{Se}  J-P. \ Serre, Local algebra, Springer (2000).

\bibitem{V3} F. \ Vaccarino, Linear representations, symmetric products and the commuting scheme, \textit{J. \ Algebra} \textbf{317}  (2007), 634-641.

\bibitem{V1} F. \ Vaccarino, Generalized symmetric functions and invariants of matrices, \textit{Math.Z.} \textbf{260} (2008), 509-526.

\bibitem{V} F. \ Vaccarino, Moduli of Linear Representations, Symmetric Products and the non Commutative Hilbert scheme in \textit{Geometric Methods in Representation Theory, II} S\'eminaires et Congress, Grenoble \textbf{24} (2010), 435-456.

\bibitem{VdB} M. \ Van den Bergh, The Brauer-Severi scheme of the trace ring of generic matrices in \textit{Perspectives in Ring Theory, Antwerp, 1987}. NATO Adv. Sci. Inst. Ser. C Math. Phys. Sci., \textbf{233}, 333-338
Kluwer Acad. Publ., Dordrecht (1988).

\end{thebibliography}

\bigskip

\begin{flushleft}


Federica~Galluzzi\\
Dipartimento di Matematica\\
Universit\`a di Torino\\
Via Carlo Alberto n.10 ,Torino\\
10123, \ ITALIA \\
e-mail: \texttt{federica.galluzzi@unito.it}\\[2ex]

Francesco~Vaccarino\\
Dipartimento di Matematica\\
Politecnico di Torino\\
C.so Duca degli Abruzzi n.24, Torino\\
 10129, \ ITALIA \\
e-mail: \texttt{francesco.vaccarino@polito.it}

\end{flushleft}

\end{document}